\documentclass[11pt, a4paper, oneside, reqno]{amsart}

\usepackage{comment}
\renewcommand{\epsilon}{\varepsilon}

\usepackage{cancel}
\usepackage[labelsep=endash]{caption}
\usepackage[utf8]{inputenc}
\usepackage{float}
\usepackage{booktabs}
\newcommand{\edit}[1]{\textcolor{black}{#1}}

\makeatletter
\def\@settitle{\begin{center}%
		\baselineskip14\p@\relax
		\normalfont\LARGE\scshape\bfseries
		\@title
	\end{center}%
}
\makeatother

\makeatletter

\def\subsection{\@startsection{subsection}{2}%
	\z@{.5\linespacing\@plus.7\linespacing}{.5\linespacing}%
	{\normalfont\large\bfseries}}

\def\subsubsection{\@startsection{subsubsection}{3}%
	\z@{.5\linespacing\@plus.7\linespacing}{.5\linespacing}%
	{\normalfont\itshape}}

\usepackage[colorlinks	= true,
			raiselinks	= true,
			linkcolor	= MidnightBlue,
			citecolor	= Mahogany,
			urlcolor	= ForestGreen,
			pdfauthor	= {Peyman Mohajerin Esfahani},
			pdftitle	= {},
			pdfkeywords	= {},   
			pdfsubject	= {},
			plainpages	= false]{hyperref}
\usepackage[usenames, dvipsnames]{color}
\usepackage{dsfont,amssymb,amsmath,subfig, graphicx,fancyhdr,mdframed}
\usepackage{amsfonts,dsfont,mathtools, mathrsfs,amsthm,wrapfig} 
\usepackage[]{siunitx}
\usepackage{ragged2e}
\usepackage{enumitem}
\usepackage{algorithm, algorithmic}
\allowdisplaybreaks
\usepackage{breqn}
\usepackage[font=small,labelfont=bf,
   justification=justified,
   format=plain]{caption}

\allowdisplaybreaks
\date{\today}

\patchcmd\@setauthors
  {\MakeUppercase{\authors}}
  {\authors}
  {}{}

\newtheorem{theorem}{Theorem}[section]

\newtheorem{lemma}[theorem]{Lemma}
\newtheorem{assumption}[theorem]{Assumption}
\newtheorem{remark}[theorem]{Remark}

\newtheorem{proposition}[theorem]{Proposition}

\graphicspath{{images/}}

\title[A Frank-Wolfe Algorithm for Strongly Monotone Variational Inequalities]{\edit{A Frank-Wolfe Algorithm for \\ Strongly Monotone Variational Inequalities}}

\author[R. {Rahimi Baghbadorani}]{Reza {Rahimi Baghbadorani}\textsuperscript{1}}
\author[P. {Mohajerin Esfahani}]{Peyman {Mohajerin Esfahani}\textsuperscript{1,2}}
\author[S. Grammatico]{Sergio Grammatico\textsuperscript{1}}

\thanks{\edit{The authors are with (1) Delft University of Technology and (2) the University of Toronto. This work was supported by the ERC grant TRUST-949796 and the NSERC Discovery grant RGPIN-2025-06544.}}

\begin{document}

\maketitle
\begin{abstract}   
We propose an accelerated algorithm with a Frank-Wolfe method as an oracle for solving strongly monotone variational inequality problems. While standard solution approaches, such as projected gradient descent (aka value iteration), involve projecting onto the desired set at each iteration, a distinctive feature of our proposed method is the use of a linear minimization oracle in each iteration. This difference potentially reduces the projection cost, a factor that can become significant for certain sets or in high-dimensional problems. We validate the performance of the proposed algorithm on the traffic assignment problem, motivated by the fact that the projection complexity per iteration increases exponentially with respect to the number of links.

\end{abstract}

\section{Introduction}\label{intro}
We consider the variational inequality problem (VIP)
\begin{equation}\label{VI-main}
\text{find} \quad x^* \in \mathcal{V} \quad \textbf{s.t.} \quad \langle g(x^*), x - x^* \rangle \geq 0, \quad \forall x \in \mathcal{V},
\end{equation}
where $\mathcal{V}$ is a compact convex set. We assume that the operator $g$ is $\mu$-strongly monotone, $L$-Lipschitz continuous, and that the solution set of \eqref{VI-main} is nonempty. The VIP is a general framework that includes several problems and applications in systems and control theory, machine learning, and operations research. For instance, both composite convex minimization and convex-concave minimax saddle point problems can be reformulated as in \eqref{VI-main} \cite{baghbadorani2024adaptive,malitsky2020golden}, which can be used in robust controller design for systems under uncertainty \cite{ball2002robust}, Nash \cite{scutari2010convex} and Stackelberg games \cite{hu2011variational}, supply chain optimization \cite{allevi2018evaluating}, and adversarial learning problems \cite{gidel2018variational}.\\
Fixed-point problems represent another important class in several domains, such as game theory and machine learning. Although it is straightforward to reformulate the VIP \eqref{VI-main} as a fixed-point problem, the converse can be computationally beneficial as discussed in \cite[section 3]{malitsky2020golden}. For instance, variational Nash equilibrium problems can be reformulated as VIPs enabling the use of efficient iterative algorithms \cite{harker1984variational, facchinei2003finite}. \\
For further applications of VIPs in systems and control theory, operations research, game theory, and machine learning, we refer interested readers to \cite{facchinei2003finite, yao1991generalized, bertsekas2009projection, baghbadorani2025douglas, pantazis2024nash} and references therein. \\
{To solve the VIP \eqref{VI-main}, several iterative algorithms have been proposed. For the sake of comparison, we review some recent and closely related methods for solving the general VIP \eqref{VI-main}, where the operator $g$ is (strongly)~monotone.}
\begin{enumerate}[label=(\roman*), itemsep = 0mm, topsep = 0mm, leftmargin = 7mm]
\item
\edit{\textbf{Projected Gradient Descent \cite{nemirovskij1983problem}:} A classical approach inspired by gradient descent in the optimization literature is  
\begin{align*}
    x_{k+1} &= \mathrm{proj}_{\mathcal{V}}\Bigl(x_k - \alpha g(x_k)\Bigr),
\end{align*}
where the operator $g(\cdot)$ replaces the gradient operator and  \(\alpha\) is the stepsize. In the literature of fixed point computation (e.g., dynamic programming), this is also known as {\em Value Iteration}. This method guarantees convergence for strongly monotone and Lipschitz operator $g$ for any stepsize $\alpha \in (0, 2\mu/L^2)$ where $\mu$ and $L$ are the strong monotonicity constant and the Lipschitz constant, respectively.}

\item
\textbf{Extragradient Descent \cite{malitsky2014extragradient}:} \edit{An improvement to the gradient descent approach is to call the operator $g$ twice in order to improve the convergence rate. This yields the algorithm}
\begin{align*}
    y_k &= \mathrm{proj}_{\mathcal{V}}\Bigl(x_k - \alpha g(x_k)\Bigr), \\
    x_{k+1} &= \mathrm{proj}_{\mathcal{V}}\Bigl(x_k - \alpha g(y_k)\Bigr),
\end{align*}
with \(\alpha\) as the stepsize. Unlike classic projected gradient descent, this method does not require strong monotonicity of the operator \(g\) and ensures convergence for a Lipschitz operator when the stepsize \(\alpha \in (0, 1/L)\).

\item
\textbf{Accelerated gradient descent \cite{nesterov2006solving}:} \edit{An influential idea in optimization, first proposed by Nesterov \cite{nesterov1983method}, is to accelerate algorithm convergence by incorporating a so-called {\em momentum} into the update dynamics. One can draw a parallel, in a similar fashion as in the Gradient Descent (i.e., replacing the gradient with the operator $g$), and arrive at }
\begin{align*}
    x_k &= \arg\max_{x \in \mathcal{V}} \, \sum_{i=0}^k \alpha_i \left[\langle g(y_i), y_i - x \rangle - \frac{\mu}{2} \|x - y_i\|^2\right], \\
    y_{k+1} &= \arg\max_{y \in \mathcal{V}} \, \langle g(x_k), x_k - y \rangle - \frac{\beta}{2} \|y - x_k\|^2,
\end{align*}
\edit{where to ensure the convergence, it suffices to choose the (time-varying) stepszie and momentum coefficients as \(\alpha_{k+1} = \frac{\mu}{L} \sum_{i=0}^k \alpha_i\) and \(\beta = L\).}

\item
\textbf{Projected Reflected Gradient Descent \cite{malitsky2015projected}:} \edit{Evaluating the operator at a reflected point, an extrapolation of the current and previous iterates, enhances stability and convergence in monotone VIP, leading to the following algorithm }
\begin{align*}
    x_{k+1} &= \mathrm{proj}_{\mathcal{V}}\Bigl(x_k - \alpha g(2x_k - x_{k-1})\Bigr),
\end{align*}
where \(\alpha \in \Big(0, (\sqrt{2}-1)/L\Big)\) is the stepsize. This method guarantees convergence for a Lipschitz operator, and unlike the extragradient method, it requires only one projection per iteration.

\item
\textbf{Golden Ratio Algorithm \cite{malitsky2020golden}:} \edit{To guarantee convergence when using the Gradient Descent method for Lipschitz but non-strongly monotone operators, negative momentum parameters are required. Introducing a negative momentum ensures convergence, leading to the algorithm}  
\begin{align*}
    y_k &= (1-\zeta)x_k + \zeta y_{k-1}, \\
    x_{k+1} &= \mathrm{proj}_{\mathcal{V}}\Bigl(y_k - \alpha g(x_k)\Bigr),
\end{align*}
with \(\alpha \in \Bigl(0, 1/(2\zeta L)\Bigr)\) as the stepsize and momentum parameter \(\zeta \in \Bigl(0, (\sqrt{5}-1)/2\Bigr]\). Additionally, overcomes the limitation of methods that rely on problem constants (like Lipschitz continuity parameter); the stepsize can be chosen adaptively, leading to the adaptive golden ratio algorithm \cite{malitsky2020golden} 
\begin{align*}
    \alpha_k &= \min\Biggl\{(\zeta + \zeta^2)\alpha_{k-1}, \frac{\|x_k - x_{k-1}\|^2}{4\zeta^2\alpha_{k-2}\|g(x_k) - g(x_{k-1})\|^2}\Biggr\}.
\end{align*}

\item
\textbf{{Operator splitting methods} \cite{facchinei2003finite}:} \edit{The operator \( g \) can be split into a summation of different operators. Solving VIP \eqref{VI-main} in these cases is equivalent to solving the fixed-point problem \(0 \in g_1(x) + g_2(x)\), where \(g(x) + \mathcal{N}_{\mathcal{V}}(x) = g_1(x) + g_2(x)\) and $\mathcal{N}_{\mathcal{V}}(x)$ is the normal cone of the compact, convex set $\mathcal{V}$ at the point $x$. The iterative update of each sub-operator leads to convergence towards the solution. A well-known class of these algorithms is the Douglas-Rachford splitting method \cite{bauschke2017correction}, which can be written as}
\begin{align*}
   y_{k+1} &= (I + \alpha g_2)^{-1}\Bigl(x_k - \alpha g_1(x_k)\Bigr), \\
   x_{k+1} &= x_k + \zeta(y_{k+1} - x_k),
\end{align*}
where the stepsize can be chosen $\alpha \in (0,1]$ and $\zeta \in (0,2)$ is a relaxation parameter. The convergence of this method is guaranteed for different cases of $g_1$ and $g_2$; we refer interested readers to~\cite{giselsson2017tight,moursi1805douglas} for further details. Additionally, we refer them to~\cite{mignoni2025monviso} for further algorithms related to applications of monotone variational inequalities and an open source Python toolbox.
\end{enumerate}
\edit{In the methods mentioned above, the projection operator onto the feasible set $\mathcal{V}$ ($\mathrm{proj}_{\mathcal{V}}(\cdot)$) is a necessary operation of the algorithm which can be costly in some cases. Frank-Wolfe (FW) is a classical approach to avoid this projection complexity by resorting to a linear oracle minimization (as opposed to the quadratic optimization of the projection operator) over the same set. Table \ref{table1} summarizes the computational complexity of the linear oracle minimization and the projection onto certain sets \cite{combettes2021complexity, braun2022conditional}.}
\begin{table}[H]
    \captionsetup{width=.95\textwidth}
    \caption{\edit{Complexity of linear minimization and projection. The parameter $\varepsilon$ denotes the precision of linear minimization or projection operator.}}
    \begin{tabular}{c c c}\hline
    {Set} & {Linear minimization} & {Projection}\\
    \hline & \\[-1.0em]
    \hspace{-3.5cm} \edit{(1) $n$-dimensional $\ell_p$-ball, $p \neq 1,2,\infty$.} & $\mathcal{O}(n)$ & $\tilde{\mathcal{O}}(n/ \varepsilon^2)$\\
    \hspace{-.15cm} \edit{(2) Nuclear norm ball of $n\times m$ matrices. $\theta$ and $\tau$ denote}\\
    \hspace{.15cm} \edit{the number of non-zero entries and the top singular}\\
    \hspace{-1.35cm} \edit{value of the projected matrix, respectively.} & $\mathcal{O}(\theta \ln{(m+n)}\sqrt{\edit{\tau}/ \varepsilon})$ & $\mathcal{O}(mn \min\{m,n\})$\\
    \hspace{-2cm} \edit{(3) Flow polytope on a graph with $m$ vertices} \\ \hspace{-1.25cm} \edit{and $n$ edges with capacity bound on edges.} & $\mathcal{O}((n \log{m})(n+m \log{m}))$ & $\mathcal{O}(n^4 \log{n})$\\
    \edit{(4) Birkhoff polytope ($n \times n$ doubly stochastic matrices).} & $\mathcal{O}(n^3)$ & $\tilde{\mathcal{O}}(n^2 / \varepsilon^2)$\\
        \\[-1.0em]
        \hline
    \end{tabular}
    \label{table1}
\end{table}
\edit{For concave-convex minimax problems and inspired by the recent works \cite{jaggi2013revisiting, lacoste2015global, hammond1984solving}, the authors of \cite{gidel17a} and \cite{chen2020efficient} prove convergence of different types of FW algorithms. They demonstrate the numerical efficiency of the FW algorithm via illustrative saddle point problems, in which projecting onto the corresponding sets is computationally demanding or even intractable; see the perfect matching problem as an example of this class of problems \cite{kolmogorov2023solving}. Motivated by this, a natural question is whether the FW algorithm can also be utilized for solving the VIP \eqref{VI-main}. To the best of our knowledge, this question is still largely unexplored, which is the focus of this study.}

\textbf{Contribution.} Following the footsteps of the accelerated Nesterov's technique for solving strongly monotone variational inequalities \cite{nesterov2006solving} and the FW algorithm for solving saddle point problems \cite{gidel17a}, we propose a novel accelerated algorithm with FW method as an oracle for solving the strongly monotone VIP \eqref{VI-main} and provide a non-asymptotic convergence rate. To validate the theoretical results, we implement the proposed method in the traffic assignment problem, \edit{an important application in transportation and operation research \cite{bertsekas1980class, bliemer2003quasi, bertsekas2009projection}}, where the complexity of the problem (the corresponding set) increases exponentially with the number of variables.

\textbf{Roadmap}. The paper is organized as follows: Section~\ref{sec2} reviews the Frank-Wolfe technique and states lemmas for solving \textit{minmax} saddle point problems. In Section~\ref{sec3}, we propose our algorithm and provide the technical proofs of the convergence theorems. Section~\ref{simulation} benchmarks the proposed algorithm in a traffic assignment application. Finally, the conclusion and future research directions are given in Section \ref{conclussion}.

\edit{\textbf{Notation.} Let $\mathcal{V}$ be a finite-dimensional real Hilbert space equipped with the standard inner product $\langle \cdot, \cdot \rangle$ and the associated norm $\|\cdot\|$ (so that we may simply write $\|x\|^2 = \langle x, x \rangle$). We define the normal cone of the set \( \mathcal{V} \) at the point \( x \in  \mathcal{V} \) as $\mathcal{N}_{\mathcal{V}} (x) = \{ u : u^\top x \geq u^\top y, \,\,  \forall y \in \mathcal{V} \}$.
The operator $\mathrm{proj}_{x \in \mathcal{V}}$ denotes the projection onto set $\mathcal{V}$ with respect to the underlying inner product norm (i.e., $\mathrm{proj}_{x \in \mathcal{V}} := \arg\min_{y \in \mathcal{V}} \|x - y\|$). We define the diameter and boundary distance of a set $\mathcal{X}$ as
 \begin{align*}
     D_{\mathcal{X}} := \sup\limits_{x, x' \in \mathcal{X}} \|x - x'\| \quad \text{and} \quad \sigma_{\mathcal{X}}(x) := \min\limits_{s \in \partial \mathcal{X}} \|x - s\|. 
 \end{align*}
The function $F(x, y)$ is $(\mu_{\mathcal{X}}, \mu_{\mathcal{Y}})$ strongly convex-concave if $F(x,y) - \frac{\mu_{\mathcal{X}}}{2}\|x\|^2 + \frac{\mu_{\mathcal{Y}}}{2}\|y\|^2$ is convex-concave. The constants $L_{XY}$ and $L_{YX}$ are the cross smooth constants of $F(x, y)$ (or equivalently the Lipschitz constants of $\nabla F$) over the set $\mathcal{X} \times \mathcal{Y}$ if for all $x,\Bar{x} \in \mathcal{X}$ and $y,\Bar{y} \in \mathcal{Y}$
\begin{align*}
    \| \nabla_x F(x,y) - \nabla_x F(x, \Bar{y})\| &\leq L_{XY} \|y - \Bar{y}\|, \\
    \| \nabla_y F(x,y) - \nabla_y F(\Bar{x}, y)\| &\leq L_{YX} \|x - \Bar{x}\|.
\end{align*}
For brevity, we refer to \( F(x, y) \) as a smooth function with constant \( L_0 = \max(L_{XY}, L_{YX}) \) if it is cross-smooth over \( \mathcal{X} \times \mathcal{Y} \).}
\section{\edit{Assumptions and Technical preliminaries}}\label{sec2}
 \edit{Before proceeding with the solution to the VIP~\eqref{VI-main}, we begin with some assumptions and lemmas that will be used throughout the paper. We note that most of the results build on the seminal work by Nesterov \cite{nesterov2006solving} and Jaggi \cite{jaggi2013revisiting}. The following assumptions hold throughout this study.
 \begin{assumption}[Operator regularity]\label{assum1}
We assume that the solution set of VIP \eqref{VI-main} is nonempty, where $\mathcal{V}$ is a compact, convex set and the operator~$g$ satisfies the following:
\begin{enumerate}[label=(\roman*), itemsep = 0mm, topsep = 0mm, leftmargin = 7mm]
    \item L-Lipschitzness: $\quad \qquad\qquad \|g(x) - g(y)\| \leq L\|x-y\|, \quad \forall x,y \in \mathcal{V}$,
    \item $\mu$-strong monotonicity: $\qquad \langle g(x) - g(y), x - y \rangle \geq \mu \|x-y\|^2, \quad \forall x,y \in \mathcal{V}.$
\end{enumerate}
\end{assumption}
Next, we start with a minimax problem, a core component of our convergence analysis for VIP~\eqref{VI-main}. Specifically, by leveraging the following minimax oracle at each iteration, we aim to eliminate the need for costly projection steps:
 \begin{align}\label{p: min-max}
     \min\limits_{x \in \mathcal{X}}\max\limits_{y \in \mathcal{Y}} F(x,y).
 \end{align}
An important concept relevant to the convergence of the \textit{minimax} problem~\eqref{p: min-max}, which helps us measure the convergence rate of \eqref{p: min-max}, is the error function, defined as
\begin{align}\label{gap-minimax}
    h_k := F(x_k, \hat{y}_k) - F(\hat{x}_k, y_k),
\end{align}
where $\hat{x}_k = \arg\max\limits_{x\in \mathcal{X}} F(x,y_k)$ and $\hat{y}_k = \arg\max\limits_{y\in \mathcal{Y}} F(x_k,y)$. The following proposition summarizes the convergence of \eqref{p: min-max} in terms of the error function $h_k$, using the FW algorithm.}
\begin{algorithm}
  \caption{FW-minimax oracle for \eqref{p: min-max} \cite{gidel17a}}
  \label{alg:FW}
  \begin{algorithmic}[1]
      \STATE Let $z_0 = (x_0, y_0) \in \mathcal{X}\times\mathcal{Y}$
      \FOR{$k=0 \ldots T$}
      \STATE Compute the partial gradients of $F$: $r_k = \begin{pmatrix} \nabla_x F(x_k, y_k) \\ -\nabla_y F(x_k, y_k) \end{pmatrix}$,
      \STATE Compute the desired direction: $s_k = \arg\min\limits_{z \in {\mathcal{X}\times\mathcal{Y}}} \, \langle z, r_k \rangle$, \\
      \STATE Compute the stepsize: $\alpha_k = \min \left(1, \frac{\nu}{2C}\langle z_k - s_k, r_k \rangle \right)$ or $\alpha_k = \frac{2}{2+k}$ \, (based on Proposition \ref{minmax lin conv})\\
      \STATE Compute the next iteration: $z_k = (1-\alpha_k)z_k + \alpha_k s_k$.
      \ENDFOR 
\end{algorithmic}
\end{algorithm}
\begin{proposition}[\edit{FW-minimax oracle convergence {\cite[Theorem~1]{gidel17a}}}] \label{minmax lin conv}
\edit{Let $F(x,y)$ be $L_0$-smooth continuous and $(\mu_{\mathcal{X}}, \mu_{\mathcal{Y}})$ strongly convex-concave on a convex, compact set $\mathcal{X}\times\mathcal{Y}$ and $(x_*, y_*)$, the solution of \eqref{p: min-max}, is in the interior of $\mathcal{X}\times\mathcal{Y}$. Consider applying the Frank-Wolfe algorithm (Algorithms \ref{alg:FW}) for saddle point problem \eqref{p: min-max} with the stepsize 
\begin{align} \label{stepsize-FW1}
    \alpha_k = \min \Big(1, \frac{\nu}{2C}\langle z_k - s_k, r_k \rangle \Big),
\end{align}
where the constants are
\begin{align*}
\rho &= \nu^2 \frac{\sigma_\mu ^2}{2C},\quad C = \frac{L_0 D_{\mathcal{X}}^2 + L_0D_{\mathcal{Y}}^2}{2},\\
\nu &= 1 - \frac{\sqrt{2}}{\sigma_\mu} \max\Bigl\{ \frac{D_{\mathcal{X}}L_{XY}}{\sqrt{\mu_{\mathcal{Y}}}}, \frac{D_{\mathcal{Y}}L_{YX}}{\sqrt{\mu_{\mathcal{X}}}}\Bigr\},\\
\sigma_\mu &= \sqrt{\min(\mu_{\mathcal{X}} \sigma^2_{\mathcal{X}}(x_*),\mu_{\mathcal{Y}} \sigma^2_{\mathcal{Y}}(y_*))}. 
\end{align*}
Then, the convergence rate for the error $h_k$ \eqref{gap-minimax} is $\mathcal{O}\left((1-\rho)^{\frac{k}{6}}\right)$.\\
Moreover, if $\sigma_\mu > 2\max\Bigl\{\frac{D_{\mathcal{X}}L_{XY}}{\mu_{\mathcal{Y}}}, \frac{D_{\mathcal{Y}}L_{YX}}{\mu_{\mathcal{X}}}\Bigr\}$ and the stepsize is set to \( \alpha_k = \frac{2}{2 + k} \), then the error \( h_k \) \eqref{gap-minimax} decreases at a sublinear rate of \( \mathcal{O}(1/k) \).}
\end{proposition}
\begin{remark}[FW-variants]
{Other FW variants, including the away-step and pairwise algorithms, achieve similar convergence guarantees under alternative assumptions to those in Proposition~\ref{minmax lin conv}, such as when the feasible set is a polytope. These variants often yield better practical performance by allowing corrections to previously chosen directions, which helps avoid flat regions and reduces zig-zagging behavior near the solution. More recently, the authors in \cite{chen2020efficient} proposed a projection-free method under assumptions similar to those in Proposition~2.2, but without requiring the solution to lie in the interior of the feasible set and their method is based on a three-loop algorithm. The proposed FW Algorithm~\ref{alg:FW} can be replaced by these alternatives. For further details, we refer interested readers to~\cite{gidel17a,chen2020efficient}.}
\end{remark}
\edit{Proposition \ref{minmax lin conv} lays the foundation for proving the convergence of the projection-free algorithm for solving VIP~\eqref{VI-main}. It avoids expensive projection steps by instead solving a strongly convex-concave \textit{minimax} problem at each iteration, which is the central motivation behind this study.}

We move forward by introducing a gap function to measure the optimality gap of VIP \eqref{VI-main} and by proposing two lemmas that are central to the development of the algorithm in this paper.\\
It is not difficult to see that, in light of the strong monotonicity of the operator $g$, the solution of VIP \eqref{VI-main}, $x^*$, satisfies the following inequality
\begin{align}
    \langle g(y), x^* - y \rangle + \frac{\mu}{2} \|y-x^*\|^2 \leq \langle g(x^*), x^* - y \rangle - \frac{\mu}{2} \|y-x^*\|^2 \leq 0, \quad \forall \, y \in \mathcal{V}.
\end{align}
In order to measure the the approximated solution of \eqref{VI-main}, we introduce a gap function $f(x)$ and the following lemma.
\begin{align}\label{gap-func}
    f(x) := \sup\limits_{y \in \mathcal{V}} \Big\{ \langle g(y), x - y \rangle + \frac{\mu}{2} \|y-x\|^2\Big\}.
\end{align}
\begin{lemma}[\edit{Gap function properties, {\cite[Theorem 1]{nesterov2006solving}}}]\label{Gap-property}
The gap function \( f(x) \) is a nonnegative, well-defined, \(\mu\)-strongly convex function on \(\mathcal{V}\) and vanishes at the unique solution of \eqref{VI-main}.
\end{lemma}
\edit{Using Lemma \ref{Gap-property} and the definition \eqref{gap-func}, our goal is to minimize \( f(x) \), which is equivalent to solving the VIP \eqref{VI-main}. To this end, let us define the following quantities}
\begin{align}\label{variable-upbound}
    S_N := \sum_{i=0}^N \lambda_i, \quad \tilde{y}_N := \frac{1}{S_N} \sum_{i=0}^N \lambda_i y_i, \quad \Delta_N := \max\limits_{x \in \mathcal{V}} \biggl\{ \sum_{i=0}^N \lambda_i \Bigl[ g(y_i), y_i - x \rangle - \frac{\mu}{2} \|x-y_i\|^2 \Bigr] \biggr\},
\end{align}
where \(\{y_i\}_{i = 0}^N \subset \mathcal{V}\) and \(\{\lambda_i\}_{i = 0}^N\) are sequences of arbitrary points and positive weights, respectively. Next lemma shows the upper bound for the gap function $f(x)$. 
\begin{lemma}[\edit{Gap function upper bound, {\cite[Lemma 1]{nesterov2006solving}}}]\label{obj-f-upb} 
Using the quantities defined in \eqref{variable-upbound}, we have the inequality
\begin{align}\label{up-ineq}
    f(\tilde{y}_N) \leq \frac{1}{S_N}\Delta_N.
\end{align}    
\end{lemma}
The proof follows from the definition of the gap function \eqref{gap-func} and the strong monotonicity of the operator \( g(x) \). For brevity, we skip the proof and refer interested readers to \cite{nesterov2006solving} for further details. 
\section{Proposed Frank-Wolfe Algorithm and Convergence Analysis}\label{sec3}
\edit{Building on the key proposition and lemmas from the previous section, this section pertains to the analysis of convergence for using the FW algorithm as an oracle in solving VIP~\eqref{VI-main}. The general form of our proposed method is provided in Algorithm \ref{alg:FW-VI}, which follows the accelerated gradient descent method~\cite{nesterov2006solving}, with the difference that the projection steps are replaced by a strongly convex-concave minimax problem solved by the FW algorithm.}

Lemmas \ref{Gap-property} and \ref{obj-f-upb} shed light on the behavior of the gap function \eqref{gap-func} and highlight the goal of minimizing and controlling the growth of $\Delta_N$.
 For $\beta > 0$, consider the functions
\begin{align*}
    \psi_{y}^{\beta}(x) := \langle g(y), y - x \rangle - \frac{\beta}{2} \|x-y\|^2, \quad \Psi_{k}(x) := \sum_{i=0}^k \lambda_i \psi_{y_i}^\mu (x).
\end{align*}
We note that $\Delta_k = \max\limits_{x \in \mathcal{V}} \Psi_{k}(x)$, and the functions \( \psi_{y}^{\beta}(x) \) and \( \Psi_{k}(x) \) are strongly concave with constants \(\beta\) and \(\mu S_k\), respectively. Consider the following iterations
\begin{equation}\label{xy iter}
\left\{
\begin{aligned}
    x_k &= \, \arg\min\limits_{x_k \in \mathcal{V}}\max\limits_{x \in \mathcal{V}} \Bigl\{W(x,x_k) - \frac{\mu}{4}S_k \|x - x_k\|^2\Bigr\}, \\ 
    y_{k+1} &= \, \arg\min\limits_{y_{k+1} \in \mathcal{V}}\max\limits_{x \in \mathcal{V}} \Bigl\{\langle -g(x_k) - \beta(y_{k+1} - x_k), x - y_{k+1} \rangle - \frac{\mu}{4}\|x - y_{k+1}\|^2 \Bigr\},
\end{aligned}
\right.
\end{equation}
where
\begin{align*}
    W(x,x_k) &:= \langle \nabla \Psi_k(x_k), x - x_k \rangle = \Bigl\langle \sum_{i=0}^k \lambda_i(-g(y_i) - \mu(x_k - y_i)), x - x_k \Bigr\rangle \\
    &= -\mu S_k \langle x_k, x - x_k \rangle +  \Bigl\langle \sum_{i=0}^k \lambda_i(-g(y_i) + \mu y_i), x - x_k \Bigr\rangle.
\end{align*}
\edit{We are now in a position to derive an upper bound for \(\Delta_k\) using iterations \eqref{xy iter}, which helps us establish the convergence of the iterates to the solution of the VIP \eqref{VI-main}.}
\begin{proposition}[Upper bound of $\Delta_k$]\label{th1-delt-up}
    If $\lambda_{k+1} \leq \frac{\mu}{2\beta}S_k$, then by using the iterates \eqref{xy iter}, we have
    \begin{align}\label{upb-Delta}
        \Delta_{k+1} \leq \Delta_k + \lambda_{k+1}\Bigl[ \frac{1}{\mu+2\beta}\|g(y_{k+1}) - g(x_k)\|^2 - \frac{\beta}{2}\|y_{k+1} - x_k\|^2 \Bigr].
    \end{align}
\end{proposition}
\begin{proof}
    We know that $\Psi_{k+1}(x) = \Psi_{k}(x) + \lambda_{k+1}\psi_{y_{k+1}}^{\mu}(x)$. Then, we have
    \begin{align}
        \Delta_{k+1} &= \max\limits_{x \in \mathcal{V}} \Bigl\{\Psi_{k}(x) + \lambda_{k+1}\psi_{y_{k+1}}^{\mu}(x)\Bigr\} \nonumber\\
        &\leq \Delta_k + \max\limits_{x\in \mathcal{V}} \Bigl\{\langle \nabla \Psi_k(x_k), x - x_k \rangle - \frac{\mu}{2}S_k \|x - x_k\|^2 +\lambda_{k+1}\psi_{y_{k+1}}^{\mu}(x)\Bigr\} \nonumber\\
        &\leq \Delta_k + \max\limits_{x\in \mathcal{V}} \Bigl\{W(x,x_k) - \frac{\mu}{4}S_k \|x - x_k\|^2\Bigr\} + \max\limits_{x\in \mathcal{V}} \Bigl\{- \frac{\mu}{4}S_k \|x - x_k\|^2 +\lambda_{k+1}\psi_{y_{k+1}}^{\mu}(x)\Bigr\}. \nonumber 
    \end{align}
    Note that, if we consider $x^\dagger = \arg\max\limits_{x \in \mathcal{V}} \Psi_k(x) = \mathrm{proj}_{x \in \mathcal{V}}\left(\frac{\sum_{i=0}^k \lambda_i\bigl(\mu y_i - g(y_i)\bigr)}{\mu S_k}\right)$, then, from the definition of the projection operator and the first-order optimality condition, we have $W(x,x^\dagger) = \langle \nabla \Psi_k(x^\dagger), x - x^\dagger \rangle \leq 0$. Therefore, we can conclude that, as indicated in Proposition \ref{minmax lin conv},  the iterates $x_k$ in \eqref{xy iter} ensures $\min\limits_{x_k \in \mathcal{V}}\max\limits_{x \in \mathcal{V}} \Bigl\{W(x,x_k) - \frac{\mu}{4}S_k \|x - x_k\|^2\Bigr\} < 0$. We further note that this strongly convex-concave minimax problem can be solved by FW algorithm with (sub)-linear convergence guarantee. Therefore, by the definition of \(x_k\) in \eqref{xy iter}, we arrive at
        \begin{align} \label{eq1-th1}
        \Delta_{k+1} &\leq \Delta_k + \overbrace{\min\limits_{x_k \in \mathcal{V}}\max\limits_{x\in \mathcal{V}} \Bigl\{W(x,x_k) - \frac{\mu}{4}S_k \|x - x_k\|^2\Bigr\}}^{\leq 0} + \max\limits_{x\in \mathcal{V}} \Bigl\{- \frac{\mu}{4}S_k \|x - x_k\|^2 +\lambda_{k+1}\psi_{y_{k+1}}^{\mu}(x)\Bigr\} \nonumber \\
        & \leq \Delta_k + \max\limits_{x\in \mathcal{V}} \Bigl\{- \frac{\mu}{4}S_k \|x - x_k\|^2 +\lambda_{k+1}\Bigl[\langle g(y_{k+1}), y_{k+1} - x \rangle - \frac{\mu}{2} \|x-y_{k+1}\|^2 \Bigr] \Bigr\}.
    \end{align}
    Now, let us analyze the second term in the ``$\max$" part of the right-hand-side of \eqref{eq1-th1}. Note that 
    \begin{align}
        \langle g(y_{k+1}), y_{k+1} - x \rangle -& \frac{\mu}{2} \|x-y_{k+1}\|^2 = \langle g(y_{k+1}) - g(x_k), y_{k+1} - x \rangle - \frac{\mu}{2} \|x-y_{k+1}\|^2 + \langle g(x_k), y_{k+1} - x \rangle, \nonumber
    \end{align}
    Similar to the previous part, if $y^\dagger$ considered as \( y^\dagger = \mathrm{proj}_{x \in \mathcal{V}}\left(x_k - \frac{g(x_k)}{\beta} \right) \), we have
    \begin{align*}
        \langle -g(x_k) - \beta(y^\dagger - x_k), x - y^\dagger \rangle \leq 0, \quad \forall x \in \mathcal{V}.
    \end{align*}
    Therefore, $y_{k+1}$ in \eqref{xy iter} ensures $\min\limits_{y_{k+1} \in \mathcal{V}}\max\limits_{x \in \mathcal{V}} \Bigl\{\langle -g(x_k) - \beta(y_{k+1} - x_k), x - y_{k+1} \rangle - \frac{\mu}{4}\|x - y_{k+1}\|^2 \Bigr\} < 0$, and using the definition of $y_{k+1}$ we can write
    \begin{align}
        \langle g(y_{k+1}), y_{k+1} - x \rangle -& \frac{\mu}{2} \|x-y_{k+1}\|^2 \nonumber\\
        &\leq \|g(y_{k+1}) - g(x_k)\| \|y_{k+1} - x\| - \frac{\mu}{4}\|y_{k+1} - x\|^2 + \beta \langle y_{k+1} - x_k, x - y_{k+1} \rangle, \nonumber
    \end{align}
    using above inequality and maximizing the right-hand-side based on $\|y_k - x\|$, one obtains
    \begin{align}
       \langle g(y_{k+1}), y_{k+1} - x \rangle - \frac{\mu}{2} \|x-y_{k+1}\|^2 \leq \frac{1}{\mu+2\beta}\|g(y_{k+1}) - g(x_k)\|^2 + \frac{\beta}{2}\|x - x_k\|^2 - \frac{\beta}{2}\|y_{k+1} - x_k\|^2. \nonumber
    \end{align}
    Putting the above inequality together with \eqref{eq1-th1}, and using the upper bound on \(\lambda_{k+1}\), we obtain the inequality \eqref{upb-Delta}.
\end{proof}
Now by setting \(\lambda_{k+1} = \frac{\mu}{2L} S_k = \frac{1}{2\gamma}S_k\) in \eqref{upb-Delta}, where \(L\) and \(\gamma = \frac{L}{\mu}\) are the Lipschitz constant and the condition number of the operator \(g\), respectively, we have \(\Delta_{k+1} < \Delta_k\). We are now ready to show the non-convergence rate of Algorithm \ref{alg:FW-VI} for solving VIP \eqref{VI-main}.
\begin{algorithm}
  \caption{Frank-Wolfe type algorithm for the VIP (FW-VIP)}
  \label{alg:FW-VI}
  \begin{algorithmic}[1]
      \STATE Given $\gamma = \frac{L}{\mu}, \lambda_0 = 1$, $\varepsilon$.
      \FOR{$k=0, \ldots, T$}
      \vspace{.2cm}
      \STATE \hspace{-.25cm}$\left\{
       \begin{array}{l}
        x_k = \arg\min\limits_{x_k \in \mathcal{V}}\max\limits_{x \in \mathcal{V}} \Bigl\{W(x,x_k) - \frac{\mu}{4}S_k \|x - x_k\|^2\Bigr\}\\
       y_{k+1} = \arg\min\limits_{y_{k+1} \in \mathcal{V}}\max\limits_{x \in \mathcal{V}} \Bigl\{\langle -g(x_k) - L(y_{k+1} - x_k), x - y_{k+1} \rangle - \frac{\mu}{4}\|x - y_{k+1}\|^2 \Bigr\}
      \end{array} \right\} {\tiny \begin{array}{c} \text{FW-minimax oracle~\ref{alg:FW}} \\ \text{with $\varepsilon$-precision} \end{array}}$
      \STATE $\lambda_{k+1} = \frac{1}{2\gamma}S_k$,
      \ENDFOR 
      \STATE \textbf{output:}\, $\tilde{y}_{k} = \frac{1}{S_k}\sum_{i=0}^k \lambda_i y_i$.
\end{algorithmic}
\end{algorithm}
\begin{theorem}[VIP-convergence via FW-minimax oracle]\label{th: vi-conv}
Consider the VIP~\eqref{VI-main} under Assumptions~\ref{assum1}, and the FW-minimax oracle~\ref{alg:FW}. Let $\gamma = L/\mu$ be the condition number of the operator $g$, and $f$ the gap function defined in~\eqref{gap-func}. Algorithm~\ref{alg:FW-VI} returns an $\varepsilon$ accurate solution to \eqref{VI-main} if the number of iterations is $T \geq (\gamma + 1) \log\left( {\gamma^2 f(x_0)}/{\varepsilon} \right)$.
\end{theorem}
Before proceeding with the proof of the theorem, we summarize the overall convergence complexity of using Algorithm~\ref{alg:FW-VI} with the FW-minimax oracle~\ref{alg:FW} as follows.
\begin{remark}[Overall complexity of VIP solution]
    The overall complexity of the solution to the VI problem~\eqref{VI-main} is the product of the complexity of the FW-minimax oracle in Algorithm~\ref{alg:FW} with the complexity of Algorithm~\ref{alg:FW-VI}, under the assumption that the solution of each sub-minimax problem lies inside the feasible domain $\mathcal{V}$. In particular, by using the stepsize~\eqref{stepsize-FW1} in Algorithm~\ref{alg:FW}, , resulting in
\begin{align*}
           \underbrace{\mathcal{O}\left(\log \left({1}/{\varepsilon}\right)\right)}_{\tiny{\text{Theorem \ref{th: vi-conv}}}} \cdot \underbrace{\mathcal{O}\left(\log \left({1}/{\varepsilon}\right)\right)}_{\tiny{\text{Proposition \ref{minmax lin conv}}}} = \mathcal{O}\left(\log^2 \left({1}/{\varepsilon}\right)\right).
\end{align*}
    Moreover, by leveraging the diminishing stepsize $\alpha_t = \tfrac{2}{2+t}$, if the extra assumption of Proposition~\ref{minmax lin conv} is satisfied, the overall complexity of solving the problem becomes 
    \begin{align*}
           \underbrace{\mathcal{O}\left(\log \left({1}/{\varepsilon}\right)\right)}_{\tiny{\text{Theorem \ref{th: vi-conv}}}} \cdot \underbrace{\mathcal{O}\left({1}/{\varepsilon}\right)}_{\tiny{\text{Proposition \ref{minmax lin conv}}}} = {\mathcal{O}}\left({\log{\left(1/\varepsilon\right)}}/{\varepsilon}\right).
\end{align*}
\end{remark}
\begin{proof}[Proof of Theorem~\ref{th: vi-conv}]
Using the definition of \(S_k\) in \eqref{variable-upbound} and the fact that \(S_0 = \lambda_0 = 1\), we can write
    \begin{align*}
        S_{k+1} = S_k + \lambda_{k+1} = \left(1+\frac{1}{\gamma}\right)S_k.
    \end{align*}
    By using Lemma \ref{obj-f-upb}, the corollary of {Proposition}~\ref{th1-delt-up}, and using the sequence $S_{k+1} = \left(1 + \frac{1}{\gamma}\right) S_k$ with $S_0 = 1$, we can conclude that $S_k = \left(1 + \frac{1}{\gamma}\right)^k$. Therefore, we obtain
    \begin{align*}
        f(\tilde{y}_{k}) \leq \frac{1}{S_k}\Delta_k \leq \frac{\Delta_0}{\left(1 + \frac{1}{\gamma}\right)^k} = \left(\frac{\gamma}{1+\gamma}\right)^k \Delta_0  = \left(1 - \frac{1}{1+\gamma}\right)^k \Delta_0 \leq \Delta_0 e^{-\frac{k}{\gamma+1}}.
    \end{align*}
    Now, we just need to approximate $\Delta_0$. Using the definition of $\Delta$ and $\Psi$, we can write
    \begin{align*}
        \Delta_0 &= \max\limits_{x \in \mathcal{V}} \Psi_0(x) = \max\limits_{x \in \mathcal{V}} \psi_{y_0}^{\mu}(x) = \max\limits_{x \in \mathcal{V}} \Bigl\{\langle g(x_0), x_0 - x \rangle -\frac{\mu}{2}\|x - x_0\|^2 \Bigr\} \\
        &= \max\limits_{x \in \mathcal{V}} \Bigl\{\langle g(x_0) - g(x^*), x_0 - x \rangle + \langle g(x^*), x_0 - x \rangle -\frac{\mu}{2}\|x - x_0\|^2 \Bigr\} \\
        &\leq \langle g(x^*), x_0 - x \rangle + \max\limits_{x \in \mathcal{V}} \Bigl\{\langle g(x_0) - g(x^*), x_0 - x \rangle -\frac{\mu}{2}\|x - x_0\|^2 \Bigr\} \leq \langle g(x^*), x_0 - x \rangle + \frac{L^2}{2\mu}\|x^* - x_0\|^2.
    \end{align*}
Then, by the definition and strong convexity of the gap function \(f\) in \eqref{gap-func}, the last inequality is less than \(f(x_0)\). Thus, the theorem's statement is concluded.
\end{proof}

\edit{We close this section by noting that the total computational time for solving the VIP \eqref{VI-main} is the product of the convergence complexity of the considered algorithm and the complexity of the mathematical operations, e.g., projection(s) used at each iteration, which can be computationally demanding in some problems. In the next section, we see this behavior and observe the effectiveness of our proposed projection-free method compared to commonly used state-of-the-art algorithms for strongly monotone VIP.} 
\section{Numerical Results}\label{simulation}
We demonstrate the performance of Algorithm \ref{alg:FW-VI} on the traffic assignment problem, an important problem studied in the transportation and operations research literature. To evaluate our performance, we compare our proposed algorithm (FW-VIP) with projected gradient descent (PGD) \cite{nemirovskij1983problem} and Nesterov's accelerated gradient method (NAG) \cite{nesterov2006solving}, two methods commonly used for solving strongly monotone VIP in the literature. Before proceeding with the simulation results, let us first briefly explain the traffic assignment problem.

\textbf{Traffic Assignment Problem (TAP):} Assume $N$ be the set of nodes, $\mathcal{E}$ be set of directed links, and $W$ denotes the set of origin-destination (OD) pairs. We assume for all $w \in W$ there is $d_w > 0$ representing traffic demand entering from origin and exiting from destination. For each OD pair $w \in W$, the demand $d_w$ is to be distributed among a given $P_w$ of a paths joining $w$. Defining $x_p$ that shows the flow carried by path $P$, we introduce the feasible path flow vector $x = \{x_p \, | \, p \in P_w, \, w\in W \}$, and the set corresponding to feasible path flow vector as
\begin{align*}
    X = \Bigl\{x\,|\, \sum_{p \in P_w} x_p = d_w, \, x_p \geq 0, \, \forall p \in P_w , \, w \in W \Bigr\}.
\end{align*}
Each collection of path flows $x \in X$ defines a collection of links flows $y_{ij}$ with
\begin{align}\label{arc-chain def}
    y_{ij} = \sum_{w \in W} \sum_{p \in P_w} \delta_p (i,j) x_p, \quad (i,j) \in \mathcal{E},
\end{align}
where $\delta_p (i,j) = 1$ if path $p$ contains link $(i,j)$ and zero otherwise. The vector of links flows $y = \{y_{ij}\, | \, (i,j) \in \mathcal{E}\}$ can be written as $y = Ax$ where $A$ is the arc-chain matrix defined by \eqref{arc-chain def}. Similarly, we can define the set of feasible link flows as $Y = AX = \Bigl\{y\,|\, y = Ax, \, x \in X \Bigr\}$.
For each link $(i,j) \in \mathcal{E}$, there is a given function $T_{ij}: \, Y \rightarrow \mathbb{R}$, which is the measure of the delay in link $(i,j)$. The vector with components $T_{ij}(y)$ is $T(y)$ and due to the choice of $T_{ij}$ it is mostly strongly monotone operator. For each $x \in X$ and corresponding $y = Ax$, the vector $T(y)$ defines the following function for each $w \in W$ and $p \in P_w$
\begin{align*}
    \Bar{T}_p (x) = \sum_{(i,j) \in \mathcal{E}} \delta_p (i,j) T_{ij}(y),
\end{align*}
which shows the total travel time of path $p$. Then, TAP is to find $x^* \in X$ such that for all $\Bar{p} \in P_w$ and $w \in W$ we have
\begin{align}\label{TAP}
    \Bar{T}_{\Bar{p}}(x^*) = \min\limits_{p \in P_w} \, \Bar{T}_{p}(x^*) \quad \text{and} \quad {x_p ^*} > 0.
\end{align}
As shown in \cite{dafermos1980traffic}, if $\Bar{T}(x)$ denotes the vector with $\Bar{T}_p (x)$ components, then $\Bar{T}(x) = A^\top T(Ax)$ and TAP \eqref{TAP} can be reformulated as the following VIP:
\begin{align}\label{TAP-VIP}
    \text{find $x^*$ such that} \quad (x - x^*)^\top \Bar{T}(x^*) \geq 0, \quad \forall x \in X.
\end{align}
Note that $A^\top T(Ax)$ is not necessarily strongly monotone unless $A^\top A$ is inevitable (as is the case in our simulation).

We consider TAP with the same structure in \cite{bertsekas2009projection} which can be viewed as model of a circular highway. Figure \ref{fig:TAP} illustrates the corresponding model, in which we consider five OD points, numbered 1--5, and five OD pairs: (1,4), (2,5), (3,1), (4,2), and (5,3), associated with the OD points. Each OD pair is linked with two possible paths (clockwise and counterclockwise). The links, types of links, time delays on each link, and the demand for each OD pair are given in Table \ref{table2}, where $h(x) = 1+x+x^2$ and we set $\kappa = 0.5$. Note that we assume flows cannot utilize the existing ramp that does not lead to its intended destination. Figure \ref{fig:VI-Gap} shows the optimality gap of the corresponding VIP with parameters reported in Table \ref{table2}. We note that, for a fair comparison and to demonstrate the effectiveness of the projection-free algorithm, we plot the optimality gap versus the total iterations multiplied by the total number of projections (in PGD or NAG) or LMO (in Algorithm \ref{alg:FW-VI}). We also note that both projection and LMO operators are evaluated using the same solver, namely \texttt{OSQP} \cite{osqp}.
\begin{table}[H]
    \centering
    \caption{Parameters in TAP.}
   \begin{tabular}{c c}\hline 
    \textbf{Types of links} & \\
    (1) Highway links: & 17, 27, 37, 47, 57, 18, 28, 38, 48, 58.\\
    \hspace{-5mm}(2) Exit ramps: & 14, 24, 34, 44, 54, 12, 22, 32, 42, 52. \\
    \hspace{3mm}(3) Entrance ramps: & 11, 21, 31, 41, 51, 13, 23, 33, 43, 53.\\
   \hspace{-3mm}(4) Bypass links: & \, 15, 25, 35, 45, 55, 16, 26, 36, 46, 56. \vspace{.2cm} \\
    \hline & \\[-1.0em]
    \textbf{Delay on links} & \\ \vspace{.1cm}
    (1) Delay on highway link $k$: & \hspace{2mm}$10\cdot h\Big[\text{flow on } k\Big] + 2\kappa \cdot h\Big[\text{flow on exit ramp from }k\Big]$.  \\ \vspace{.1cm}
    \hspace{-5mm}(2) Delay on exit ramp $k$: & \hspace{-64.8mm} $h\Big[\text{flow on }k\Big]$. \\ \vspace{.1cm}
    \hspace{3mm}(3) Delay on entrance ramp $k$: & \hspace{10mm} $h\Big[\text{flow on }k\Big] + \kappa \cdot h\Big[\text{flow on bypass link merging with } k\Big]$.  \\
    \hspace{-3mm}(4) Delay on bypass link $k$: & \hspace{-63mm} $h\Big[\text{flow on }k\Big]$.  \vspace{.2cm}\\
    \hline & \\[-1.0em]
    \textbf{Demands on OD pairs} & $d(1,4) = 0.1$, $d(2,5) = 0.2$, $d(3,1) = 0.3$, $d(4,2) = 0.4$, $d(5,3) = 0.5$.
       \\[.5em]
        \hline
    \end{tabular}
   \label{table2}
\end{table}
\begin{minipage}{\linewidth}
    \centering
    \captionsetup{justification=centering}
    \begin{minipage}{0.40\linewidth}
        \begin{figure}[H]
        \includegraphics[scale=0.38]{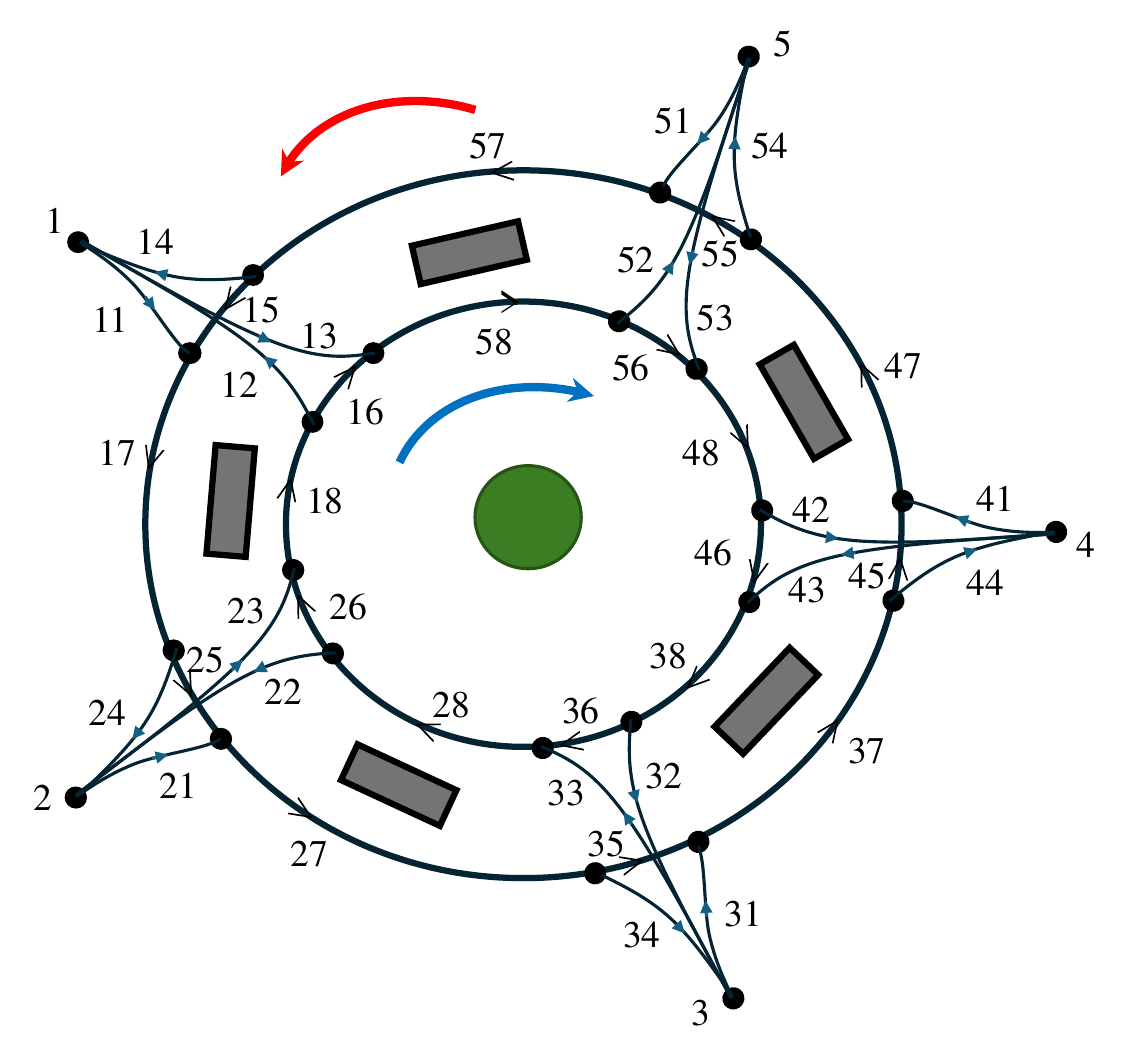}
        \caption{TAP model (Figure 2 in \cite{bertsekas2009projection}).}
        \label{fig:TAP}
    \end{figure}
   \end{minipage}
      \hspace{0.05\linewidth}
\begin{minipage}{0.5\linewidth}
\centering
\captionsetup{justification=centering}
    \begin{figure}[H]
    \includegraphics[scale=0.40]{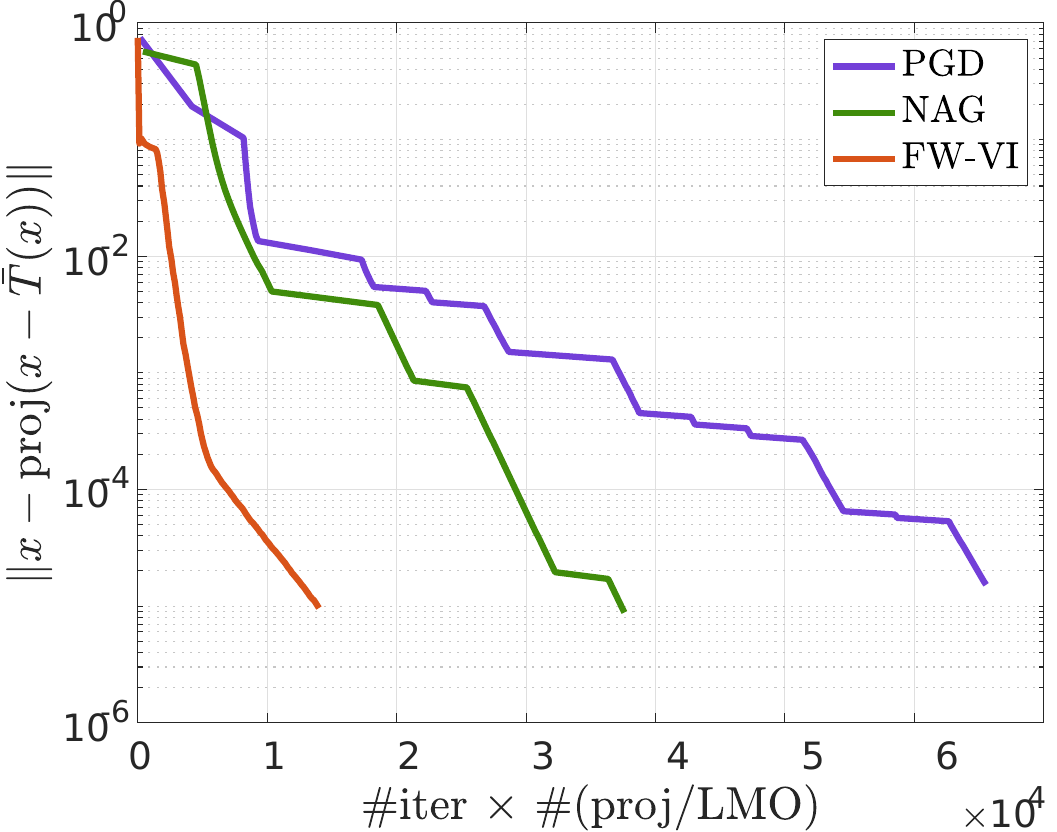}
    \caption{Optimality gap of the VIP in the TAP.}
    \label{fig:VI-Gap}
    \end{figure}
    \end{minipage}
\end{minipage}
\section{Further Discussion, Limitations, and Future Directions}\label{conclussion}
In this section, we provide additional insights concerning the proposed Frank-Wolfe technique for solving strongly monotone variational inequalities~(1), its limitations, and potential future directions.

\begin{enumerate}[label=(\roman*), itemsep = 0mm, topsep = 0mm, leftmargin = 7mm]

\item \textbf{Further insights into the Frank-Wolfe technique for solving VIP~\eqref{VI-main}:} The main motivation behind the Frank-Wolfe technique is to avoid expensive projections in each iteration and to select a more suitable descent direction compared to projected methods, which may exhibit zig-zagging behavior for certain constraint sets. This can lead to lower computational complexity and potentially faster convergence. Furthermore, as stated in Proposition~\ref{th1-delt-up}, we only require the negativity of 
\begin{align*}
N(x,y) := \min\limits_{x_k \in \mathcal{V}}\max\limits_{x \in \mathcal{V}} \Bigl\{ W(x,x_k) - \tfrac{\mu}{4} S_k \|x - x_k\|^2 \Bigr\},
\end{align*}
to guarantee the convergence of Algorithm~\ref{alg:FW-VI}. While this negativity is ensured by the worst-case convergence bound in Proposition~\ref{minmax lin conv}, in practice the inner Frank-Wolfe problem can be terminated as soon as $N(x,y)$ becomes negative, significantly reducing computational complexity.

\item \textbf{Extension to monotone variational inequalities:} Solving the VIP using a projection-free algorithm requires solving a minimax subproblem. Unfortunately, existing projection-free methods for minimax problems rely on strong convexity-concavity of the underlying function, which in our setting can be ensured by assuming strong monotonicity of the operator in the VIP. Developing a Frank-Wolfe method for minimax problems without this assumption is a key step toward extending these methods to general monotone variational inequality problems and is an important direction for future work.

\item \textbf{Relation to prior works and alternative assumptions:} Based on the structure of the VIP, alternative assumptions can be applied for the Frank-Wolfe algorithm in solving the inner minimax problems, replacing the requirement in Proposition~\ref{minmax lin conv} that the solution lies strictly within the domain:
\begin{itemize}
    \item The sets $\mathcal{X}$ and $\mathcal{Y}$ are polytopes (case~(II) in Theorem~1 \cite{gidel17a}): In this case, the interior-point assumption is not required. Instead, we can use the away-step Frank-Wolfe algorithm (Algorithm~3 in \cite{gidel17a}) and the concept of pyramidal width (Eq.~9 in \cite{lacoste2015global}) instead of the border distance $\sigma_{\mathcal{X}}$, while still guaranteeing the same convergence as in Proposition~\ref{minmax lin conv}.
    \item More recently, \cite{chen2020efficient} proposed a projection-free algorithm (Algorithm~3) under a similar assumption to Proposition~\ref{minmax lin conv}, without requiring the solution to lie in the interior of $\mathcal{X} \times \mathcal{Y}$. This algorithm requires $\tilde{\mathcal{O}}(1/\sqrt{\varepsilon})$ first-order and $\tilde{\mathcal{O}}(1/\varepsilon^2)$ linear optimization oracle calls to achieve $\varepsilon$-precision for the saddle-point problem~\eqref{p: min-max}. Its main limitation is implementation complexity, as it involves three nested loops per iteration.
\end{itemize}

\item \textbf{Limitations and future directions:}
\begin{itemize}
    \item The primary limitation of the proposed method is the presence of a secondary loop in each iteration for solving the minimax problem. In the context of the Frank-Wolfe algorithm for minimizing a smooth convex function, two key inequalities, smoothness and convexity, are typically used. For the VIP, no explicit value function exists, making these inequalities difficult to apply. Although we can define a convex gap function $f(x)$ \eqref{gap-func}, whose minimization is equivalent to solving the VIP, this function is parametric, complicating the use of standard convex optimization inequalities. Restricting the domain to specific sets and leveraging duality may allow for a single-loop algorithm, which is an important future research direction. Adaptive or decreasing accuracy requirements for the inner loop may also help eliminate the secondary loop.
    
    \item Another limitation is the algorithm's dependence on problem-specific constants. Future extensions could involve adaptive stepsize methods \cite{malitsky2020golden}, which approximate Lipschitz or strong monotonicity constants to accelerate convergence.

    \item Recent works, such as \cite{liu2022newton}, establish Newton-type methods combined with Frank-Wolfe techniques for constrained self-concordant minimization, and \cite{lin2024perseus} propose high-order methods for solving the VIP. Developing high-order methods for the VIP using Frank-Wolfe techniques represents another promising research direction.
\end{itemize}
\end{enumerate}
\bibliographystyle{plain}
\bibliography{references.bib}
\end{document}